\documentclass[english,11pt,reqno]{smfart}

\usepackage{amsthm}
\usepackage{amsmath}
\usepackage{amsfonts}
\usepackage{amstext}
\usepackage[latin1]{inputenc}
\usepackage{amscd}
\usepackage{latexsym}
\usepackage{bm}
\usepackage{amssymb}
\usepackage{euscript}
\usepackage{a4wide}
\usepackage{amsmath,amssymb,graphicx}
\usepackage{amssymb}
\usepackage{amsmath}
\usepackage{mathrsfs,mathtools}

\newtheorem{theorem}{Theorem}
\newtheorem{proposition}{Proposition}
\newtheorem{lemma}{Lemma}
\newtheorem{corollary}{Corollary}
\newtheorem{example}{Example}
\newtheorem{definition}{Definition}

\def\di{\displaystyle}

\newcommand{\N}{\mathbb{N}}
\newcommand{\R}{\mathbb{R}}
\newcommand{\LL}{\mathcal{L}}
\newcommand{\CC}{\mathscr{C}}
\newcommand{\Cc}{\mathscr{C}^\infty_{\mathrm{c}}}
\newcommand{\PP}{\mathscr{P}}

\newcommand{\W}{\mathrm{W}}
\renewcommand{\L}{\mathrm{L}}
\newcommand{\E}{\mathrm{E}}

\newcommand*{\hooktwoheadrightarrow}{\lhook\joinrel\twoheadrightarrow}
\newcommand{\fonction}[5]{\begin{array}[t]{lrcl}#1 :&#2 &\longrightarrow &#3\\&#4& \longmapsto &#5 \end{array}}

\newcommand{\fonctionsansdef}[3]{\begin{array}[t]{lrcl}#1 :&#2 &\longrightarrow &#3 \end{array}}

\begin{document}

\thanks{This is a preprint of a paper whose final and definite form will appear 
in \emph{Differential and Integral Equations}, ISSN 0893-4983 
(See http://www.aftabi.com/DIE.html). Submitted 19/July/2013; Accepted 16/March/2014.}

\title[]{Existence of minimizers for generalized Lagrangian functionals
and a necessary optimality condition --- Application to fractional variational problems}

\author{Lo\"ic Bourdin}

\address{Laboratoire de Math\'ematiques et de leurs Applications - Pau (LMAP).
UMR CNRS 5142. Universit\'e de Pau et des Pays de l'Adour, France.}
\email{bourdin.l@etud.univ-pau.fr}

\author{Tatiana Odzijewicz}
\address{Center for Research and Development in Mathematics and Applications,
Department of Mathematics, University of Aveiro, 3810-193 Aveiro, Portugal.}
\email{tatianao@ua.pt}

\author{Delfim F.M. Torres}
\address{Center for Research and Development in Mathematics and Applications,
Department of Mathematics, University of Aveiro, 3810-193 Aveiro, Portugal.}
\email{delfim@ua.pt}

\maketitle


\begin{abstract}
We study dynamic minimization problems of the calculus of variations with generalized
Lagrangian functionals that depend on a general linear operator $K$ and defined on
bounded-time intervals. Under assumptions of regularity, convexity and coercivity,
we derive sufficient conditions ensuring the existence of a minimizer. Finally,
we obtain necessary optimality conditions of Euler--Lagrange type. Main results
are illustrated with special cases, when $K$ is a general kernel operator and,
in particular, with $K$ the fractional integral of Riemann--Liouville and Hadamard.
The application of our results to the recent fractional calculus of variations gives answer
to an open question posed in [Abstr. Appl. Anal. 2012, Art. ID 871912; doi:10.1155/2012/871912].
\end{abstract}



\tableofcontents


\section{Introduction}

The mathematical field that deals with derivatives of any real order is called fractional calculus.
For a long time, it was only considered as a pure mathematical branch. Nevertheless, during the last two decades,
fractional calculus has attracted the attention of many researchers and it has been successfully applied
in various areas like computational biology \cite{magi} or economy \cite{comt}. In particular, the first
and well-established application of fractional operators was in the physical context of anomalous diffusion,
see \cite{neel,neel2} for example. Here we can mention \cite{metz}, demonstrating that fractional equations
work as a complementary tool in the description of anomalous transport processes. Let us refer to \cite{hilf3}
for a general review of the applications of fractional calculus in several fields of Physics. In a more general
point of view, fractional differential equations are even considered as an alternative model to non-linear differential
equations, see \cite{boni}. \\

Recently, a subtopic of the fractional calculus gains importance: the calculus of variations with Lagrangian
functionals involving fractional derivatives. This leads to the statement of fractional Euler--Lagrange equations,
see \cite{alme,bale2,bale3}. This idea was introduced by Riewe in 1996-97 \cite{riew,riew2} in view of finding
fractional variational structures for non conservative differential equations. One can find a similar and more
conclusive reasoning in \cite{cres8,cres7}. For the state of the art on the fractional calculus of variations,
we refer the reader to the recent book \cite{torr5}. For optimal control problems with stochastic equations driven
by fractional noise, see \cite{dunc} and references therein. \\

Fractional Euler--Lagrange equations characterize the critical points of fractional Lagrangian functionals
and consequently, they are necessary optimality conditions for optimizers. Nevertheless, despite particular
results in \cite{jiao,klim}, no general existence results of an optimizer are provided in the literature.
This is a reason why we have provided in \cite{bour6,bour9} sufficient conditions ensuring the existence
of a minimizer for fractional Lagrangian functionals in the Riemann--Liouville and Caputo senses.
Let us remind that, in these two previous papers, the method developed is widely inspired from \cite{cesa,daco2}
where general existence results of a minimizer for classical Lagrangian functionals are provided. \\

There exist many notions of fractional integrals and derivatives. We can cite the notions of Riemann--Liouville,
Hadamard, Caputo and Gr\"unwald--Letnikov, see \cite{kilb,podl,samk}. In consequence, there exist a lot of versions
of fractional Euler--Lagrange equations. An unifying perspective to the subject is possible by considering general
linear operators, like kernel operators \cite{kiry,odzi,odzi2}. In \cite{odzi,odzi2}, authors are then interested
in the calculus of variations with Lagrangian functionals involving general operators. This leads to the statement
of generalized Euler--Lagrange equations. Unfortunately, once again, no general existence results are provided
for this unifying framework. \\

Our aim in this paper is then to give sufficient conditions ensuring the existence of a minimizer for generalized
Lagrangian functionals in the case of bounded-time intervals. We also prove a necessary optimality condition
of Euler--Lagrange type. Finally, we illustrate our results by special cases of general kernel operators and,
in particular, of fractional integrals (Riemann--Liouville and Hadamard). \\

The paper is organized as follows. In Section~\ref{section1}, sufficient conditions  ensuring the existence
of a minimizer for a generalized Lagrangian functional are derived. We first establish a Tonelli-type theorem
with general sufficient conditions in Section~\ref{section11}. Then, we give more concrete ones
in Sections~\ref{section12} and \ref{section13}. In Section~\ref{section2}, we prove a necessary optimality
condition of Euler--Lagrange type. Section~\ref{section3} is devoted to examples of general kernel operators.
In particular, we study the cases of fractional integrals of Riemann--Liouville (fixed and variable order)
and of Hadamard. Finally, in Section~\ref{section4}, we provide some improvements to the results
of Section~\ref{section1} by modifying some assumptions. In Section~\ref{section5} of conclusion,
we give some perspectives of possible generalizations.


\section{Existence of minimizers for a generalized Lagrangian functional}
\label{section1}

Let us consider $a < b$ two reals, let $d \in \N^*$ be the dimension and let $\Vert \cdot \Vert$
denote the usual Euclidean norm of $\R^d$. Let us denote by:
\begin{itemize}
\item $\CC := \CC ([a,b];\R^d)$ the usual space of continuous
functions endowed with its usual norm $\Vert \cdot \Vert_\infty$;
\item $\Cc := \Cc ([a,b];\R^d)$ the usual space of infinitely
differentiable functions compactly supported in $]a,b[$;
\end{itemize}
and, for any $1 \leq r \leq \infty$, let us denote by:
\begin{itemize}
\item $\L^r := \L^r (a,b;\R^d)$ the usual space of $r$-Lebesgue integrable
functions endowed with its usual norm $\Vert \cdot \Vert_{\L^r}$;
\item $\W^{1,r} := \W^{1,r} (a,b;\R^d)$ the usual $r$-Sobolev space endowed
with its usual norm $\Vert \cdot \Vert_{\W^{1,r}}$.
\end{itemize}
Let us remind that the compact embedding $\W^{1,r} \hooktwoheadrightarrow \CC$
holds for any $1 < r \leq \infty$, see \cite{brez} for a detailed proof. \\

In the whole paper, let us consider $1 < p < \infty$ (resp. $1 < q < \infty$)
and let $p'$ (resp. $q'$) denote the adjoint of $p$ (resp. $q$) \textit{i.e.}
$p'=p/(p-1)$ (resp. $q'=q/(q-1)$). In this section, our aim is to give sufficient
conditions ensuring the existence of a minimizer
for the following generalized Lagrangian functional:
\begin{equation}
\fonction{\LL}{\E}{\R}{u}{\di \int_a^b L(u,K[u],\dot{u},K[\dot{u}],t) \; dt ,}
\end{equation}
where $\E$ is a weakly closed subset of $\W^{1,p}$, $\dot{u}$ denotes the derivative of $u$,
$K$ is a linear bounded operator from $\L^p$ to $\L^q$ and $L$ is a Lagrangian of class $\CC^1$:
\begin{equation}
\fonction{L}{(\R^d)^4 \times [a,b]}{\R}{(x_1,x_2,x_3,x_4,t)}{L(x_1,x_2,x_3,x_4,t).}
\end{equation}
For any $i=1,2,3,4$, let us denote by $\partial_i L$ the partial derivative
of $L$ with respect to its $i$th variable. \\

Let us remind that, in this paper, $K$ is destined to play the role of a general kernel
operator and more precisely of a fractional integral
(Riemann--Liouville or Hadamard), see Section~\ref{section3}.


\subsection{A Tonelli-type theorem}
\label{section11}

In this section, we state a Tonelli-type theorem ensuring the existence of a minimizer
for $\LL$ with the help of general assumptions of regularity, coercivity and convexity.
These three hypothesis are usual in the classical case, see \cite{cesa,daco2}. Precisely:

\begin{definition}
\label{defreg}
A Lagrangian $L$ is said to be \emph{regular} if it satisfies:
\begin{itemize}
\item $L (u,K[u],\dot{u},K[\dot{u}],t) \in \L^1$;
\item $\partial_1 L (u,K[u],\dot{u},K[\dot{u}],t) \in \L^1$;
\item $\partial_2 L (u,K[u],\dot{u},K[\dot{u}],t) \in \L^{q'}$;
\item $\partial_3 L (u,K[u],\dot{u},K[\dot{u}],t) \in \L^{p'}$;
\item $\partial_4 L (u,K[u],\dot{u},K[\dot{u}],t) \in \L^{q'}$,
\end{itemize}
for any $u \in \W^{1,p}$.
\end{definition}

\begin{definition}
\label{defcoer}
A Lagrangian functional $\LL$ is said to be \emph{coercive} on $\E$ if it satisfies:
\begin{equation}
\lim\limits_{\substack{\Vert u \Vert_{\W^{1,p}} \to \infty \\ u \in \E }} \LL (u) = +\infty.
\end{equation}
\end{definition}

We are now in position to state the following general result:

\begin{theorem}[Tonelli-type theorem]
\label{thmtonelli}
Let us assume that:
\begin{itemize}
\item $L$ is regular;
\item $\LL$ is coercive on $\E$;
\item $L(\cdot,t)$ is convex on $(\R^d)^4$ for any $t \in [a,b]$.
\end{itemize}
Then, there exists a minimizer for $\LL$ on $\E$.
\end{theorem}

\begin{proof}
Since $L$ is regular, $L (u,K[u],\dot{u},K[\dot{u}],t) \in \L^1$ and then $\LL (u)$
exists in $\R$ for any $u \in \E$. Let us introduce a minimizing sequence
$(u_n)_{n \in \N} \subset \E$ satisfying:
\begin{equation}
\LL (u_n) \longrightarrow \inf\limits_{u \in \E} \LL (u) < +\infty.
\end{equation}
Since $\LL$ is coercive, $(u_n)_{n \in \N}$ is bounded in $\W^{1,p}$. Since $\W^{1,p}$
is a reflexive Banach space, it exists a subsequence of $(u_n)_{n \in \N}$
weakly convergent in $\W^{1,p}$. In the following, we still denote this subsequence
by $(u_n)_{n \in \N}$ and we denote by $\bar{u}$ its weak limit. Since $\E$
is a weakly closed subset of $\W^{1,p}$, $\bar{u} \in \E$. Finally,
using the convexity of $L$, we have for any $n \in \N$:
\begin{equation}
\label{eq00}
\LL (u_n) \geq \LL (\bar{u}) + \di \int_a^b \partial_1 L \cdot (u_n-\bar{u})
+ \partial_2 L \cdot \left(K[u_n]-K[\bar{u}]\right) + \partial_3 L
\cdot (\dot{u}_n-\dot{\bar{u}}) + \partial_4 L \cdot (K[\dot{u}_n]-K[\dot{\bar{u}}])  \; dt,
\end{equation}
where $\partial_i L$ are taken in $(\bar{u},K[\bar{u}],\dot{\bar{u}},K[\dot{\bar{u}}],t)$ for any $i=1,2,3,4$. \\

Now, from these four following facts:
\begin{itemize}
\item $L$ is regular;
\item $u_n \xrightharpoonup[]{\W^{1,p}} \bar{u}$;
\item $K$ is linear bounded from $\L^p$ to $\L^q$;
\item the compact embedding $\W^{1,p} \hooktwoheadrightarrow \CC$ holds;
\end{itemize}
one can easily conclude that:
\begin{itemize}
\item $\partial_3 L(\bar{u},K[\bar{u}],\dot{\bar{u}},K[\dot{\bar{u}}],t)
\in \L^{p'}$ and $\dot{u_n} \xrightharpoonup[]{\L^p} \dot{\bar{u}}$;
\item $\partial_4 L(\bar{u},K[\bar{u}],\dot{\bar{u}},K[\dot{\bar{u}}],t)
\in \L^{q'}$ and $K[\dot{u_n}] \xrightharpoonup[]{\L^q} K[\dot{\bar{u}}]$;
\item $\partial_1 L(\bar{u},K[\bar{u}],\dot{\bar{u}},K[\dot{\bar{u}}],t)
\in \L^1$ and $u_n \xrightarrow[]{\L^\infty} \bar{u}$;
\item $\partial_2 L(\bar{u},K[\bar{u}],\dot{\bar{u}},K[\dot{\bar{u}}],t)
\in \L^{q'}$ and $K[u_n] \xrightarrow[]{\L^q} K[\bar{u}]$.
\end{itemize}
Finally, taking $n \to \infty$ in inequality \eqref{eq00}, we obtain:
\begin{equation}
\inf\limits_{u \in \E} \LL (u) \geq \LL (\bar{u}) \in \R,
\end{equation}
which completes the proof.
\end{proof}

The first two hypothesis of Theorem~\ref{thmtonelli} are very general. Consequently,
in Sections~\ref{section12} and \ref{section13}, we give concrete assumptions
on $L$ ensuring its regularity and the coercivity of $\LL$. \\

The last hypothesis of convexity is strong. Nevertheless, from more regularity assumptions
on $L$ and on $K$, we prove in Section~\ref{section4} that we can provide versions
of Theorem~\ref{thmtonelli} with weaker convexity assumptions.


\subsection{Sufficient condition for a regular Lagrangian $L$}
\label{section12}

In this section, we give a sufficient condition on $L$ implying its regularity.
First, for any $M \geq 1$, let us define the set $\PP_M$ of maps
$\fonctionsansdef{P}{(\R^d)^4 \times [a,b]}{\R^+}$ such that for any
$(x_1,x_2,x_3,x_4,t) \in (\R^d)^4 \times [a,b]$:
\begin{equation}
P(x_1,x_2,x_3,x_4,t) = \di \sum_{k=0}^{N} c_k(x_1,t) \Vert x_2
\Vert^{d_{2,k}} \Vert x_3 \Vert^{d_{3,k}} \Vert x_4 \Vert^{d_{4,k}},
\end{equation}
with $ N \in \N$ and where, for any $k=0,\ldots,N$, $\fonctionsansdef{c_k}{\R^d \times [a,b]}{\R^+}$
is continuous and $(d_{2,k},d_{3,k},d_{4,k}) \in [0,q] \times [0,p] \times [0,q]$ satisfies
$d_{2,k}+ (q/p) d_{3,k} + d_{4,k} \leq (q/M)$. \\

The following lemma shows the interest of sets $\PP_M$:

\begin{lemma}
\label{lemregP}
Let $M \geq 1$ and $P \in \PP_M$. Then:
\begin{equation}
\forall u \in \W^{1,p}, \; P(u,K[u],\dot{u},K[\dot{u}],t) \in \L^{M}.
\end{equation}
\end{lemma}

\begin{proof}
For any $k=0,\ldots,N$, $c_k(u,t)$ is continuous and then is in $\L^\infty$. Furthermore,
$\Vert K[u] \Vert^{d_{2,k}} \in \L^{q/d_{2,k}}$, $\Vert \dot{u} \Vert^{d_{3,k}} \in \L^{p/d_{3,k}}$
and $\Vert K[\dot{u}] \Vert^{d_{4,k}} \in \L^{q/d_{4,k}}$. Consequently:
\begin{equation}
c_k(u,t) \Vert K[u] \Vert^{d_{2,k}}  \Vert \dot{u} \Vert^{d_{3,k}}
\Vert  K[\dot{u}] \Vert^{d_{4,k}}  \in \L^r,
\end{equation}
with $r=q/(d_{2,k} + (q/p) d_{3,k} +d_{4,k}) \geq M$. The proof is complete.
\end{proof}

Finally, from Lemma~\ref{lemregP}, one can easily obtain the following proposition:

\begin{proposition}
\label{prop1}
If there exist $P_0 \in \PP_{1}$, $P_1 \in \PP_{1}$, $P_2 \in \PP_{q'}$,
$P_3 \in \PP_{p'}$ and $P_4 \in \PP_{q'}$ such that:
\begin{itemize}
\item $ \vert L(x_1,x_2,x_3,x_4,t) \vert \leq P_0(x_1,x_2,x_3,x_4,t) $;
\item $ \Vert \partial_1 L(x_1,x_2,x_3,x_4,t) \Vert \leq P_1(x_1,x_2,x_3,x_4,t) $;
\item $ \Vert \partial_2 L(x_1,x_2,x_3,x_4,t) \Vert \leq P_2(x_1,x_2,x_3,x_4,t) $;
\item $ \Vert \partial_3 L(x_1,x_2,x_3,x_4,t) \Vert \leq P_3(x_1,x_2,x_3,x_4,t) $;
\item $ \Vert \partial_4 L(x_1,x_2,x_3,x_4,t) \Vert \leq P_4(x_1,x_2,x_3,x_4,t) $,
\end{itemize}
for any $(x_1,x_2,x_3,x_4,t) \in (\R^d)^4 \times [a,b]$, then $L$ is regular.
\end{proposition}

This last proposition states that if the norms of $L$ and of its partial derivatives
are controlled from above by elements of $\PP_M$, then $L$ is regular.
We will see some examples in Section~\ref{section14}.


\subsection{Sufficient condition for a coercive Lagrangian functional $\LL$}
\label{section13}

The definition of coercivity for a Lagrangian functional $\LL$ is strongly dependent
on the considered set $\E$. Consequently, in this section, we will consider an example
of set $\E$ and we will give a sufficient condition
on $L$ ensuring the coercivity of $\LL$ in this case. \\

Precisely, let us consider $u_0 \in \R^d$ and $\E = \W^{1,p}_a$ where
$\W^{1,p}_a := \{ u \in \W^{1,p}, \; u(a)=u_0 \}$. From the compact embedding
$\W^{1,p} \hooktwoheadrightarrow \CC $, $\W^{1,p}_a$ is a weakly closed subset of $\W^{1,p}$. \\

An important consequence of such a choice of set $\E$ is given by the following lemma:

\begin{lemma}
\label{lemdom}
There exist $A_0$, $A_1 \geq 0$ such that for any $u \in \W^{1,p}_a$:
\begin{itemize}
\item $\Vert u \Vert_{\L^\infty} \leq A_0 \Vert \dot{u} \Vert_{\L^p} + A_1$;
\item $\Vert K[u] \Vert_{\L^q} \leq A_0 \Vert \dot{u} \Vert_{\L^p} + A_1$;
\item $\Vert K[\dot{u}] \Vert_{\L^q} \leq A_0 \Vert \dot{u} \Vert_{\L^p} + A_1$.
\end{itemize}
\end{lemma}

\begin{proof}
The last inequality comes from the boundedness of $K$. Let us consider the second one.
For any $u \in \W^{1,p}_a$, we have $\Vert u \Vert_{\L^p} \leq \Vert u-u_0 \Vert_{\L^p}
+ \Vert u_0 \Vert_{\L^p} \leq (b-a) \Vert \dot{u} \Vert_{\L^p} + (b-a)^{1/p} \Vert u_0 \Vert$.
We conclude using again the boundedness of $K$. Now, let us consider the first inequality.
For any $u \in \W^{1,p}_a$, we have $\Vert u \Vert_{\L^\infty} \leq \Vert u-u_0 \Vert_{\L^\infty}
+ \Vert u_0 \Vert \leq \Vert \dot{u} \Vert_{\L^1} + \Vert u_0 \Vert \leq (b-a)^{1/p'} \Vert \dot{u}
\Vert_{\L^p} + \Vert u_0 \Vert$. Finally, we have just to define $A_0$ and $A_1$ as the maxima
of the appearing constants. The proof is complete.
\end{proof}

Precisely, this lemma states the \textit{affine domination} of the term $\Vert \dot{u} \Vert_{\L^p}$
on the terms $\Vert u \Vert_{\L^\infty}$, $\Vert K[u] \Vert_{\L^q}$ and $\Vert K[\dot{u}] \Vert_{\L^q}$
for any $u \in \W^{1,p}_a$. This characteristic of $\W^{1,p}_a$ leads us to give the following sufficient
condition for a coercive Lagrangian functional $\LL$:

\begin{proposition}
\label{prop1b}
Assume that for any $(x_1,x_2,x_3,x_4,t) \in (\R^d)^4 \times [a,b]$:
\begin{equation}
L(x_1,x_2,x_3,x_4,t) \geq c_0 \Vert x_3 \Vert^{p} + \di \sum_{k=1}^{N} c_k
\Vert x_1 \Vert^{d_{1,k}} \Vert x_2 \Vert^{d_{2,k}}
\Vert x_3 \Vert^{d_{3,k}} \Vert x_4 \Vert^{d_{4,k}},
\end{equation}
with $c_0 > 0$ and $N \in \N^*$ and where, for any $k=1,\ldots,N$, $c_k \in \R$
and $(d_{1,k},d_{2,k},d_{3,k},d_{4,k}) \in \R^+ \times [0,q] \times [0,p] \times [0,q]$ satisfies:
\begin{equation}
d_{2,k}+(q/p)d_{3,k}+d_{4,k} \leq q \quad \text{and} \quad d_{1,k}+d_{2,k}+d_{3,k}+d_{4,k} < p.
\end{equation}
Then $\LL$ is coercive on $\W^{1,p}_a$.
\end{proposition}

\begin{proof}
Let us define $r_k = q/(d_{2,k}+(q/p)d_{3,k}+d_{4,k}) \geq 1$ and let $r'_k$ denote the adjoint
of $r_k$ \textit{i.e.} $r'_k = r_k/(r_k - 1)$. Using H\"older's inequality,
one can easily prove that, for any $u \in \W^{1,p}_a$, we have:
\begin{equation}
\LL (u) \geq c_0 \Vert \dot{u} \Vert_{\L^p}^p - \di \sum_{k=1}^N \vert
c_k \vert (b-a)^{1/r_k'} \Vert u \Vert^{d_{1,k}}_{\L^\infty} \Vert K[u]
\Vert^{d_{2,k}}_{\L^q}  \Vert \dot{u} \Vert^{d_{3,k}}_{\L^p}
\Vert K[\dot{u}] \Vert^{d_{4,k}}_{\L^q}.
\end{equation}
From the \textit{affine domination} of $\Vert \dot{u} \Vert_{\L^p}$
(see Lemma~\ref{lemdom}) and from the assumption
$d_{1,k}+d_{2,k}+d_{3,k}+d_{4,k} < p$, we obtain that:
\begin{equation}
\lim\limits_{\substack{\Vert \dot{u} \Vert_{\L^{p}} \to \infty \\
u \in \W^{1,p}_a }} \LL (u) = +\infty.
\end{equation}
Finally, from Lemma~\ref{lemdom}, we also have in $\W^{1,p}_a$:
\begin{equation}
\Vert \dot{u} \Vert_{\L^{p}} \to \infty
\Longleftrightarrow \Vert u \Vert_{\W^{1,p}} \to \infty .
\end{equation}
Consequently, $\LL$ is coercive on $\W^{1,p}_a$. The proof is complete.
\end{proof}

In this section, we have studied the case where $\E$ is the weakly closed subset
of $\W^{1,p}$ satisfying the initial condition $u(a)=u_0$. For other examples
of set $\E$, let us note that all the results of this section are still valid when:
\begin{itemize}
\item $\E$ is the weakly closed subset of $\W^{1,p}$ satisfying a final condition in $t=b$;
\item $\E$ is the weakly closed subset of $\W^{1,p}$ satisfying two boundary conditions in $t=a$ and in $t=b$.
\end{itemize}

For more general examples of set $\E$, one has to deduce the following reasoning.
A structure of $\E$ implying the \textit{domination} of one of terms $u$, $K[u]$,
$\dot{u}$ or $K[\dot{u}]$ has to be associated to a Lagrangian controlled
from below by a map preserving this domination.


\subsection{Examples of Lagrangian $L$}
\label{section14}

In this section, we give several examples of a convex Lagrangian $L$ satisfying assumptions
of Propositions~\ref{prop1} and \ref{prop1b}. In consequence, they are examples of application
of Theorem~\ref{thmtonelli} in the case $\E = \W^{1,p}_a$.

\begin{example}
\label{ex1}
The most classical examples of a Lagrangian are the quadratic ones.
Let us consider the following one:
\begin{equation}
L(x_1,x_2,x_3,x_4,t) = c(t)+\dfrac{1}{2} \di \sum_{i=1}^4 \Vert x_i \Vert^2,
\end{equation}
where $\fonctionsansdef{c}{[a,b]}{\R}$ is of class $\CC^1$. One can easily check
that $L$ satisfies the assumptions of Propositions~\ref{prop1} and \ref{prop1b}
with $p =2$ and $q \geq 2$. Moreover, $L$ satisfies the convexity hypothesis
of Theorem~\ref{thmtonelli}. Consequently, for any linear operator $K$ bounded
from $L^2$ to $L^q$, one can conclude that there exists a minimizer
of $\LL$ defined on $\W^{1,2}_a$.
\end{example}

\begin{example}
\label{ex2}
Let us consider $p=2$ and $q \geq 2$ and let us still denote $L$ the Lagrangian defined
in Example~\ref{ex1}. To obtain a more general example, one can define a Lagrangian
$L_1$ from $L$ as a time-dependent homothetic transformation and/or translation of its variables. Precisely:
\begin{equation}
\label{eqtransf}
L_1(x_1,x_2,x_3,x_4,t) = L(c_1(t) x_1+ c^0_1(t),c_2(t) x_2
+ c^0_2(t),c_3(t) x_3+ c^0_3(t),c_4(t) x_4+ c^0_4(t),t),
\end{equation}
where $\fonctionsansdef{c_i}{[a,b]}{\R}$ and $\fonctionsansdef{c^0_i}{[a,b]}{\R^d}$
are of class $\CC^1$ for any $i=1,2,3,4$. In this case, $L_1$ also satisfies
the convexity hypothesis of Theorem~\ref{thmtonelli} and the assumptions
of Proposition~\ref{prop1}. Moreover, if $c_3$ is with values in $\R^+$,
then $L_1$ also satisfies the assumption of Proposition~\ref{prop1b}.
\end{example}

One should be careful: this last remark is not available in a more general context.
Precisely, if a general Lagrangian $L$ satisfies the convexity hypothesis
of Theorem~\ref{thmtonelli} and assumptions of Propositions~\ref{prop1}
and \ref{prop1b}, then a Lagrangian $L_1$ obtained by \eqref{eqtransf} also satisfies
the convexity hypothesis of Theorem~\ref{thmtonelli} and the assumptions
of Proposition~\ref{prop1}. Nevertheless, the assumption
of Proposition~\ref{prop1b} can be lost by this process.

\begin{example}
\label{ex3}
We can also study quasi-linear examples given by a Lagrangian of the type
\begin{equation}
L(x_1,x_2,x_3,x_4,t) = c(t)+\dfrac{1}{p} \Vert x_3 \Vert^p
+ \di \sum_{i=1}^4 f_i (t) \cdot x_i,
\end{equation}
where $\fonctionsansdef{c}{[a,b]}{\R}$ and for any $i=1,2,3,4$,
$\fonctionsansdef{f_i}{[a,b]}{\R^d}$ are of class $\CC^1$. In this case,
$L$ satisfies the assumptions of Propositions~\ref{prop1} and \ref{prop1b}
for any $1 < p < \infty$ and $1 < q < \infty$. Consequently, since $L$ satisfies
the convexity hypothesis of Theorem~\ref{thmtonelli}, for any linear operator
$K$ bounded from $L^p$ to $L^q$, one can conclude that there exists
a minimizer of $\LL$ defined on $\W^{1,p}_a$.
\end{example}

The most important constraint in order to apply Theorem~\ref{thmtonelli}
is the convexity hypothesis. This is the reason why the previous examples
concern convex quasi-polynomial Lagrangians. Nevertheless, in Section~\ref{section4},
we are going to provide some improved versions of Theorem~\ref{thmtonelli}
with weaker convexity assumptions. This will be allowed by more regularity hypothesis
on $L$ and/or on $K$. We refer to Section~\ref{section4} for more details.


\section{Necessary optimality condition for a minimizer}
\label{section2}

Throughout this section, we assume additionally that:
\begin{itemize}
\item $L$ satisfies the assumptions of Proposition~\ref{prop1}
(in particular, $L$ is regular and $\LL (u)$ exists in $\R$ for any $u \in \E$);
\item $\E$ satisfies the following condition:
\begin{equation}
\forall u \in \E, \; \forall v \in \Cc, \; \exists 0 < \varepsilon \leq 1,
\; \forall \vert h \vert \leq \varepsilon, \; u+hv \in \E.
\end{equation}
\end{itemize}
This last assumption is satisfied if $\E+\Cc \subset \E$
(for example $\E = \W^{1,p}_a$ in Section~\ref{section13}). \\

Let us remind that $\LL$ is said to be \emph{differentiable}
at a point $u \in \E$ in the $\Cc$-direction if the following map:
\begin{equation}
\fonction{D\LL (u)}{\Cc}{\R}{v}{D\LL (u)(v)
:= \lim\limits_{h \to 0} \dfrac{\LL (u+hv)- \LL(u)}{h}}
\end{equation}
is well-defined. In this case, $u$ is moreover said to be a \emph{critical point}
of $\LL$ (in the $\Cc$-direction sense) if $D\LL (u) = 0$. \\

We characterize the critical points of $\LL$ as the weak solutions
of a generalized Euler--Lagrange equation. In particular, a necessary condition
for a point $u \in \E$ to be a minimizer of $\LL$ is to be a weak solution
of this generalized Euler--Lagrange equation. \\

Let us precise that \textit{weak solution} has to be understood
as solution of the equation almost everywhere on $(a,b)$.


\subsection{Differentiability of $\LL$ in the $\Cc$-direction}

Before proving the differentiability of $\LL$
in the $\Cc$-direction, we state the following lemma:

\begin{lemma}
\label{lemdomP}
Let $M \geq 1$ and $P \in \PP_M$. Then, for any $u \in \E$ and any $v \in \Cc$,
it exists $g \in \L^M (a,b;\R^+)$ such that for any $ h \in [-\varepsilon,\varepsilon]$:
\begin{equation}
P(u+hv,K[u]+hK[v],\dot{u}+h\dot{v},K[\dot{u}]+hK[\dot{v}],t ) \leq g.
\end{equation}
\end{lemma}

\begin{proof}
Indeed, for any $k=0,\ldots,N$, for almost all $t \in (a,b)$
and for any $ h \in [-\varepsilon,\varepsilon]$, we have:
\begin{multline}
\label{eqproof}
c_k(u(t)+hv(t),t) \Vert K[u](t) + h K[v](t) \Vert^{d_{2,k}} \Vert \dot{u}(t)
+h \dot{v}(t) \Vert^{d_{3,k}} \Vert K[\dot{u}](t) + h K[\dot{v}](t) \Vert^{d_{4,k}}\\
\leq \bar{c}_k  (\underbrace{\Vert K[u](t) \Vert^{d_{2,k}}
+ \Vert K[v](t) \Vert^{d_{2,k}}}_{\in \L^{q/d_{2,k}}})(\underbrace{\Vert \dot{u}(t) \Vert^{d_{3,k}}
+ \Vert \dot{v}(t) \Vert^{d_{3,k}}}_{\in \L^{p/d_{3,k}}})(\underbrace{\Vert K[\dot{u}](t) \Vert^{d_{4,k}}
+ \Vert K[\dot{v}](t) \Vert^{d_{4,k}}}_{\in \L^{q/d_{4,k}}}),
\end{multline}
where $\bar{c}_k = 2^{d_{2,k}+d_{3,k}+d_{4,k}} \max\limits_{[a,b]\times[-\varepsilon,\varepsilon]}
c_k(u(t)+hv(t),t)$ exists in $\R$ because $c_k$, $u$ and $v$ are continuous.
Since $d_{2,k} + (q/p) d_{3,k} +d_{4,k} \leq (q/M)$, the right-hand side of inequality \eqref{eqproof}
is in $\L^M (a,b;\R^+)$ and is independent of $h$. The proof is complete.
\end{proof}

From this previous result, we can prove:

\begin{proposition}
\label{prop2}
Let us assume that $L$ satisfies the assumptions of Proposition~\ref{prop1}. Then,
$\LL$ is differentiable in the $\Cc$-direction at any point $u \in \E$. Moreover:
\begin{equation}
\forall u \in \E, \; \forall v \in \Cc, \; D\LL (u)(v)=\di \int_a^b \partial_1 L \cdot v
+ \partial_2 L \cdot K[v] + \partial_3 L \cdot \dot{v} + \partial_4 L \cdot K[\dot{v}] \; dt,
\end{equation}
where $\partial_i L$ are taken in $(u,K[u],\dot{u},K[\dot{u}],t)$ for any $i=1,2,3,4$.
\end{proposition}

\begin{proof}
Let $u \in \E$ and $v \in \Cc$. Let us define:
\begin{equation}
\psi_{u,v} (t,h) : = L (u(t)+hv(t),K[u](t)+hK[v](t),\dot{u}(t)
+h\dot{v}(t),K[\dot{u}](t)+hK[\dot{v}](t),t ),
\end{equation}
for any $\vert h \vert \leq \varepsilon$ and for almost every $t \in (a,b)$.
Then, let us define the following map:
\begin{equation}
\fonction{\phi_{u,v}}{[-\varepsilon,\varepsilon]}{\R}{h}{\LL (u+hv)
= \di \int_a^b \psi_{u,v} (t,h) \; dt.}
\end{equation}
Our aim is to prove that the following term:
\begin{equation}
D\LL (u)(v) = \lim\limits_{h \to 0} \dfrac{\LL (u+hv) - \LL (u)}{h}
= \lim\limits_{h \to 0} \dfrac{\phi_{u,v} (h) - \phi_{u,v}(0)}{h} = \phi_{u,v}'(0)
\end{equation}
exists in $\R$. In order to differentiate $\phi_{u,v}$, we use the theorem of differentiation
under the integral sign. Indeed, we have for almost all $t \in (a,b)$ that $\psi_{u,v}(t,\cdot)$
is differentiable on $[-\varepsilon,\varepsilon]$ with
\begin{multline}
\dfrac{\partial \psi_{u,v}}{\partial h} (t,h) = \partial_1 L(\star_h) \cdot v(t)
+ \partial_2 L(\star_h)\cdot K[v](t) + \partial_3 L(\star_h) \cdot \dot{v}(t)
+ \partial_4 L(\star_h) \cdot K[\dot{v}](t),
\end{multline}
where $\star_h = (u(t)+hv(t),K[u](t)+hK[v](t),\dot{u}(t)+h\dot{v}(t),K[\dot{u}](t)+hK[\dot{v}](t),t )$.
Then, since $L$ satisfies the assumptions of Proposition~\ref{prop1} and from Lemma~\ref{lemdomP},
there exist $g_1 \in \L^1 (a,b;\R^+)$, $g_2 \in \L^{q'} (a,b;\R^+)$, $g_3 \in \L^{p'} (a,b;\R^+)$
and $g_4 \in \L^{q'} (a,b;\R^+)$ such that for any $h \in [-\varepsilon,\varepsilon]$
and for almost all $t \in (a,b)$:
\begin{equation}
\label{eqproof2}
\left\vert \dfrac{\partial \psi_{u,v}}{\partial h} (t,h) \right\vert \leq g_1 (t) \Vert v(t) \Vert
+ g_2 (t) \Vert K[v](t) \Vert + g_3 (t) \Vert \dot{v}(t) \Vert + g_4 (t) \Vert K[\dot{v}](t) \Vert.
\end{equation}
Since $v \in \L^\infty$, $K[v] \in \L^q$, $\dot{v} \in \L^p$ and $K[\dot{v}] \in \L^q$,
we can conclude that the right-hand side of inequality \eqref{eqproof2} is in $\L^1 (a,b;\R^+)$
and is independent of $h$. Consequently, we can use the theorem of differentiation under the integral
sign and we obtain that $\phi_{u,v}$ is differentiable with:
\begin{eqnarray}
\forall h \in [-\varepsilon,\varepsilon], \; \phi_{u,v}'(h)
= \di \int_a^b \dfrac{\partial \psi_{u,v}}{\partial h} (t,h) \; dt .
\end{eqnarray}
The proof is completed by taking $h=0$ in the previous equality.
\end{proof}


\subsection{Generalized Euler--Lagrange equation}

Let us give a characterization of the critical points of $\LL$. In this way,
let us introduce $\fonctionsansdef{K^*}{\L^{q'}}{\L^{p'}}$ the adjoint operator of $K$ satisfying:
\begin{equation}
\forall u_1 \in \L^{q'}, \; \forall u_2 \in \L^{p}, \; \di \int_a^b u_1 \cdot K[u_2] \; dt
= \di \int_a^b K^*[u_1] \cdot u_2 \; dt.
\end{equation}
Let us remind that the existence and the uniqueness of $K^*$ is provided by the classical Riesz theorem.
Using this adjoint operator, we can prove the following result:

\begin{theorem}
\label{thmgel}
Let us assume that $L$ satisfies the assumptions of Proposition~\ref{prop1} and let $u \in \E$.
Then, $u$ is a critical point of $\LL$ if and only if $u$ is a weak solution of the following
\emph{generalized Euler--Lagrange equation}:
\begin{equation}
\label{gel}\tag{GEL}
\dfrac{d}{dt} \big( \partial_3 L + K^* [\partial_4 L] \big) = \partial_1 L+K^* [\partial_2 L],
\end{equation}
where $\partial_i L$ are taken in $(u,K[u],\dot{u},K[\dot{u}],t)$ for any $i=1,2,3,4$.
\end{theorem}

\begin{proof}
Let $u \in \E$. Then, from Proposition~\ref{prop2}, we have for any $v \in \Cc$:
\begin{eqnarray}
D\LL (u)(v) & = & \di \int_a^b \partial_1 L \cdot v + \partial_2 L \cdot K[v]
+ \partial_3 L \cdot \dot{v} + \partial_4 L \cdot K[\dot{v}] \; dt \\
& = & \di \int_a^b \big( \partial_1 L + K^* [ \partial_2 L ] \big) \cdot v
+ \big( \partial_3 L + K^* [\partial_4 L] \big) \cdot \dot{v} \; dt.
\end{eqnarray}
Then, taking an absolutely continuous anti-derivative $w_u$ of $\partial_1 L
+ K^* [ \partial_2 L ] \in \L^1$, we obtain by integration by parts that:
\begin{equation}
\label{eqproof3}
D\LL (u)(v) = \di \int_a^b \big( \partial_3 L
+ K^* [\partial_4 L] - w_u \big) \cdot \dot{v} \; dt.
\end{equation}
From definition, $u$ is a critical point of $\LL$ if and only if $D\LL (u)(v) = 0$
for any $v \in \Cc$. Consequently, from equality \eqref{eqproof3}, $u$ is a critical
point of $\LL$ if and only if there exists a constant $C \in \R^d$ such that
for almost all $t \in (a,b)$, we have:
\begin{equation}
\label{eqproof333}
\partial_3 L + K^* [\partial_4 L] = C + w_u.
\end{equation}
Since the right-hand side of \eqref{eqproof333} is absolutely continuous, we can differentiate
it almost everywhere on $(a,b)$. Finally, we obtain that $u$ is a critical point of $\LL$
if and only if the following equation holds almost everywhere on $(a,b)$:
\begin{equation}
\dfrac{d}{dt} \big( \partial_3 L + K^* [\partial_4 L] \big) = \partial_1 L+K^* [\partial_2 L].
\end{equation}
The proof is complete.
\end{proof}

Finally, combining Theorems~\ref{thmtonelli} and \ref{thmgel}, we prove the following corollary
stating a necessary optimality condition for a minimizer of $\LL$:

\begin{corollary}
Let us assume that $L$ satisfies the assumptions of Proposition~\ref{prop1}, $\LL$ is coercive
on $E$ and $L(\cdot,t)$ is convex on $(\R^d)^4$ for any $t \in [a,b]$. Then, the minimizer
$\bar{u}$ of $\LL$ (given by Theorem~\ref{thmtonelli}) is a weak solution
of the generalized Euler--Lagrange equation \eqref{gel}.
\end{corollary}

\begin{proof}
Indeed, since $L$ satisfies the assumptions of Proposition~\ref{prop1}, $L$ is regular.
Consequently, from Theorem~\ref{thmtonelli}, we know that $\LL$ admits a minimizer $\bar{u} \in \E$.
In particular, $\bar{u}$ is a critical point of $\LL$. Finally, from Theorem~\ref{thmgel},
$\bar{u}$ is a weak solution of \eqref{gel}.
\end{proof}


\section{Application to kernel operators $K$}
\label{section3}

In Sections~\ref{section1} and \ref{section2}, the general assumption made on the operator $K$
is totally independent of the considered set $\E$ and considered Lagrangian $L$. Then, we can
give general examples independent of these two elements. \\

Precisely, this paper is devoted to general kernel operators used in \cite{kiry,odzi,odzi2},
see Section~\ref{section31}. Let us note that fractional integrals of Riemann--Liouville
and Hadamard are particular examples of kernel operators,
see Sections~\ref{section32},~\ref{section32b} and \ref{section33}.


\subsection{General kernel operators}
\label{section31}

Let us define the following triangle:
\begin{equation}
\Delta := \{ (t,x) \in \R^2, \; a \leq x < t \leq b \}
\end{equation}
and let us consider $k$ a function defined almost everywhere on $\Delta$ with values
in $\R$. For any function $f$ defined almost everywhere on $(a,b)$
with values in $\R^d$, let us define for almost all $t \in (a,b)$:
\begin{equation}
K[f](t) = \lambda_1 \di \int_a^t k(t,y) f(y) \; dy
+ \lambda_2 \int_t^b k(y,t) f(y) \; dy,
\end{equation}
with $\lambda_1$, $\lambda_2 \in \R$.
Operator $K$ is said to be a kernel operator. \\

Assuming regularity of the kernel $k$, we can prove the following result:

\begin{proposition}
\label{prop3}
Let us assume that $q \geq p'$ and $k \in \L^q (\Delta;\R)$.
Then, $K$ is a linear bounded operator from $\L^p$ to $\L^q$.
\end{proposition}

\begin{proof}
The linearity is obvious. Then, let us prove that $K$ is bounded from $\L^p$ to $\L^q$.
Considering only the first term, let us prove that the following inequality holds for any $ f \in \L^p$:
\begin{equation}
\label{eq01}
\left( \di \int_a^b \left\Vert \di \int_a^t k(t,y) f(y) \; dy \right\Vert^q \; dt \right)^{1/q}
\leq (b-a)^{(1/p')-(1/q)} \Vert k \Vert_{\L^q (\Delta,\R)} \Vert f \Vert_{\L^p}.
\end{equation}
Since $q \geq p'$ and using Fubini's theorem, we have $k(t,\cdot) \in \L^q (a,t;\R)
\subset \L^{p'}(a,t;\R)$ for almost all $t \in (a,b)$. Then, using two times H\"older's
inequality, we have for almost all $t \in (a,b)$:
\begin{equation}\label{eq02}
\left\Vert \di \int_a^t k(t,y) f(y) \; dy \right\Vert^q
\leq \left( \di \int_a^t \vert k(t,y) \vert^{p'} \; dy
\right)^{q/p'} \Vert f \Vert_{\L^p}^q \leq (b-a)^{(q/p')-1}
\di \int_a^t \vert k(t,y) \vert^q \; dy \; \Vert f \Vert_{\L^p}^q.
\end{equation}
Hence, integrating equation \eqref{eq02} on the interval $(a,b)$,
we obtain inequality \eqref{eq01}. The proof is completed using
the same strategy on the second term in the definition of $K$.
\end{proof}

In the special case $q=p'$, let us explicit the value of $K^*$:

\begin{proposition}
\label{prop4}
Let us assume that $q=p'$ and $k \in \L^q (\Delta;\R)$. Then,
the operator $K^*$ defined for any $f \in \L^{q'}$ and almost all $t \in (a,b)$ by:
\begin{equation}
\label{eq03}
K^* [f](t) = \lambda_2 \di \int_a^t k(t,y) f(y) \; dy + \lambda_1 \int_t^b k(y,t) f(y) \; dy
\end{equation}
is a linear bounded operator from $\L^{q'}$ to $\L^{p'}$. Moreover, $K^*$ is the adjoint operator of $K$.
\end{proposition}

\begin{proof}
Since $q = p'$ and using Proposition~\ref{prop3}, $K$ is a linear bounded operator from $\L^p$ to $\L^q$.
Exchanging the roles of $p$ and $q'$ and exchanging the roles of $q$ and $p'$ in Proposition~\ref{prop3},
we obtain that $K^*$ is a linear bounded operator from $\L^{q'}$ to $\L^{p'}$. The second part is easily
proved using Fubini's theorem. Indeed, considering only the first term of the definition of $K$,
the following inequality holds for any $u_1 \in \L^{q'}$ and any $u_2 \in \L^p$:
\begin{equation}
\di \int_a^b u_1 (t) \cdot \int_a^t k(t,y) u_2(y) \; dy \; dt
= \di \int_a^b u_2 (y) \cdot \int_y^b k(t,y) u_1(t) \; dt \; dy.
\end{equation}
The proof is completed by using the same strategy on the second term of the definition of $K$.
\end{proof}

In the case of a general kernel operator $K$ associated to a kernel
$k \in \L^q(\Delta;\R) $ with $q=p'$, let us define the following operators:
\begin{equation}
A := \dfrac{d}{dt} \circ K, \quad B:= K \circ \dfrac{d}{dt}, \quad A^*
:= \dfrac{d}{dt} \circ K^* \quad \text{and} \quad B^*:= K^* \circ \dfrac{d}{dt}.
\end{equation}
Then, the generalized Lagrangian functional $\LL$ can be written as:
\begin{equation}
\fonction{\LL}{\E}{\R}{u}{\di \int_a^b L(u,K[u],\dot{u},B[u],t) \; dt.}
\end{equation}
We then recover the generalized Lagrangian functional $\LL$ studied in \cite{odzi}
where the existence of a minimizer is posed as an open question. Let us assume additionally
that $\E$ and $L$ satisfy the assumptions of Section~\ref{section2}. If a solution $u \in \E$
of the generalized Euler--Lagrange equation \eqref{gel} is sufficiently regular
(in order to make $\partial_3 L$ and $K^* [\partial_4 L]$ absolutely continuous),
then \eqref{gel} in $u$ can be written as:
\begin{equation}
\dfrac{d}{dt} \big( \partial_3 L \big) + A^* [\partial_4 L]
= \partial_1 L+K^* [\partial_2 L].
\end{equation}
Hence, we recover the generalized Euler--Lagrange equation proved in \cite{odzi}.


\subsection{The fractional integrals of Riemann--Liouville}
\label{section32}

In this section, we assume that $q=p$. For any $0 < \alpha < 1$, we denote by $K^\alpha$
the kernel operator associated to $k^\alpha (t,x) = (t-x)^{\alpha -1} / \Gamma (\alpha)$.
In this case, $K^\alpha$ corresponds to the operator $\lambda_1 I^\alpha_{a+}
+ \lambda_2 I^\alpha_{b-}$ where $I^\alpha_{a+}$ (resp. $I^\alpha_{b-}$) denotes
the left (resp. right) fractional integral of Riemann--Liouville of order $\alpha$.
We refer to \cite{kilb,samk} for details proving that:
\begin{itemize}
\item $I^\alpha_{a+}$ and $I^\alpha_{b-}$ are linear bounded operators from $\L^p$ to $\L^p$;
\item $I^\alpha_{a+}$ is the adjoint operator of $I^\alpha_{b-}$ (and conversely).
\end{itemize}
Consequently, $K^\alpha$ is a linear bounded operator from $\L^p$ to $\L^p$ and $K^*$
is given by $\lambda_2 I^\alpha_{a+} + \lambda_1 I^\alpha_{b-}$. \\

Let us remind that the common left and right fractional derivatives
of Riemann--Liouville (resp. of Caputo) of order $\alpha$ are respectively given by:
\begin{multline}
D^\alpha_{a+} = \dfrac{d}{dt} \circ I^{1-\alpha}_{a+} \; \text{and} \; D^\alpha_{b-}
= - \dfrac{d}{dt} \circ I^{1-\alpha}_{b-} \\ \left( \text{resp.} \; {}_{\mathrm{c}} D^\alpha_{a+}
=  I^{1-\alpha}_{a+} \circ \dfrac{d}{dt} \; \text{and} \; {}_{\mathrm{c}} D^\alpha_{b-}
=  - I^{1-\alpha}_{b-} \circ \dfrac{d}{dt} \right).
\end{multline}
Finally, in the particular case $K=K^{1-\alpha}$ and $(\lambda_1,\lambda_2) = (1,0)$,
Section~\ref{section1} recovers the case of the following fractional Lagrangian functional:
\begin{equation}
\fonction{\LL}{\E}{\R}{u}{\di \int_a^b L(u,I^{1-\alpha}_{a+}[u],
\dot{u},{}_{\mathrm{c}} D^\alpha_{a+} u,t) \; dt,}
\end{equation}
studied in \cite{bour9}. Let us assume additionally that $\E$ and $L$ satisfy
the assumptions of Section~\ref{section2}. If a solution $u \in \E$
of the generalized Euler--Lagrange equation \eqref{gel} is sufficiently regular
(in order to make $\partial_3 L$ and $I^{1-\alpha}_{b-} [\partial_4 L]$ absolutely continuous),
then \eqref{gel} along $u$ can be written as the following fractional Euler--Lagrange equation:
\begin{equation}
\dfrac{d}{dt} \big( \partial_3 L \big) - D^{\alpha}_{b-} [\partial_4 L]
= \partial_1 L+I^{1-\alpha}_{b-} [\partial_2 L].
\end{equation}


\subsection{The fractional integrals of Riemann--Liouville with variable order}
\label{section32b}

In this section, we assume that $q=p'$. For any map $\fonctionsansdef{\alpha}{\Delta}{[\delta,1]}$
with $\delta > (1/p)$, we denote by $K^\alpha$ the kernel operator associated
to $k^\alpha (t,x) = (t-x)^{\alpha (t,x) -1} / \Gamma (\alpha (t,x))$. In this case,
$K^\alpha$ corresponds to the operator $\lambda_1 I^\alpha_{a+} + \lambda_2 I^\alpha_{b-}$
where $I^\alpha_{a+}$ (resp. $I^\alpha_{b-}$) denotes the left (resp. right) fractional integral
of Riemann--Liouville with variable order $\alpha$, see \cite{lore,odzi3,samk2}. In this section,
we have just to prove that $k^\alpha \in \L^q (\Delta,\R)$ in order to use the results
of Section~\ref{section31}. Let us note that since $\alpha$ is with values in $[\delta,1]$
with $\delta > 0$, then $1 / (\Gamma \circ \alpha)$ is bounded. Hence, we have just to prove that
$(\Gamma \circ \alpha ) k^\alpha \in \L^q (\Delta,\R)$. We have two different cases: $b-a \leq 1$ and $b-a > 1$. \\

In the first case, for any $(t,x) \in \Delta$, we have $0 < t-x \leq 1$ and $q(\delta -1) > -1$. Then:
\begin{equation}
\di \int_a^t (t-x)^{q(\alpha (t,x) -1)} \; dx \leq \di \int_a^t (t-x)^{q(\delta -1)} \; dx
= \dfrac{(t-a)^{q(\delta -1)+1}}{q(\delta -1)+1} \leq \dfrac{1}{q(\delta -1)+1}.
\end{equation}
In the second case, for almost all $(t,x) \in \Delta \cap (a,a+1) \times (a,b)$,
we have $0 < t-x \leq 1$. Consequently, we conclude in the same way that:
\begin{equation}
\di \int_a^t (t-x)^{q(\alpha (t,x) -1)} \; dx \leq \dfrac{1}{q(\delta -1)+1}.
\end{equation}
Still in the second case, for almost all $(t,x) \in \Delta \cap (a+1,b) \times (a,b)$,
we have $x < t-1$ or $t-1 \leq x \leq t$. Then:
\begin{eqnarray}
\di \int_a^t (t-x)^{q(\alpha (t,x) -1)} \; dx & = & \di \int_a^{t-1} (t-x)^{q(\alpha (t,x) -1)} \; dx
+ \di \int_{t-1}^t (t-x)^{q(\alpha (t,x) -1)} \; dx \\
& \leq & b-a-1 + \dfrac{1}{q(\delta -1)+1}.
\end{eqnarray}
Consequently, in any case, there exists a constant $C \in \R$ such that for almost all $t \in (a,b)$:
\begin{equation}
\di \int_a^t \vert k^\alpha (t,x) \vert^q \; dx \leq C \in \L^1 (a,b;\R).
\end{equation}
Finally, $k^\alpha \in \L^q (\Delta,\R)$. From Section~\ref{section31}, $K$ is then a linear bounded
operator from $\L^p$ to $ \L^q$ and its adjoint operator is given by:
\begin{equation}
K^* = \lambda_2 I^\alpha_{a+} + \lambda_1 I^\alpha_{b-}.
\end{equation}
Then, we can apply the same strategy as in Section~\ref{section32} in order to recover the case
of a fractional Lagrangian functional involving fractional derivatives of Caputo with variable order
and to retrieve the associated fractional Euler--Lagrange equation.


\subsection{The fractional integrals of Hadamard}
\label{section33}

In this section, we assume that $a >0$ and $q=p$. For any $0 < \alpha < 1$,
we denote by $K^\alpha$ the kernel operator associated to $k^\alpha(t,x) = \log^{\alpha -1} (t/x) / x$.
In this case, $K^\alpha$ corresponds to the operator $\lambda_1 J^\alpha_{a+} + \lambda_2 J^\alpha_{b-}$
where $J^\alpha_{a+}$ (resp. $J^\alpha_{b-}$) denotes the left (resp. right) fractional integral of Hadamard
of order $\alpha$. We refer to \cite{kilb,samk} for details proving that:
\begin{itemize}
\item $J^\alpha_{a+}$ and $J^\alpha_{b-}$ are linear bounded operators from $\L^p$ to $\L^p$;
\item $J^\alpha_{a+}$ is the adjoint operator of $J^\alpha_{b-}$ (and conversely).
\end{itemize}
Consequently, $K^\alpha$ is a linear bounded operator from $\L^p$ to $\L^p$ and $K^*$
is given by $\lambda_2 J^\alpha_{a+} + \lambda_1 J^\alpha_{b-}$. \\

Let us remind that the common left and right fractional derivatives of Hadamard
(resp. of Caputo-Hadamard) of order $\alpha$ are respectively given by:
\begin{multline}
D^\alpha_{a+} = \dfrac{d}{dt} \circ J^{1-\alpha}_{a+} \; \text{and} \; D^\alpha_{b-}
= - \dfrac{d}{dt} \circ J^{1-\alpha}_{b-} \\ \left( \text{resp.} \; {}_{\mathrm{c}} D^\alpha_{a+}
= J^{1-\alpha}_{a+} \circ \dfrac{d}{dt} \; \text{and} \; {}_{\mathrm{c}} D^\alpha_{b-}
= - J^{1-\alpha}_{b-} \circ \dfrac{d}{dt} \right).
\end{multline}
In the particular case $K=K^{1-\alpha}$ and $(\lambda_1,\lambda_2) = (0,-1)$,
we get from Section~\ref{section1} the case of the following fractional Lagrangian functional:
\begin{equation}
\fonction{\LL}{\E}{\R}{u}{\di \int_a^b L(u,-J^{1-\alpha}_{b-}[u],
\dot{u},{}_{\mathrm{c}} D^\alpha_{b-} u,t) \; dt.}
\end{equation}
Let us assume additionally that $\E$ and $L$ satisfy the assumptions of Section~\ref{section2}.
If a solution $u \in \E$ of the generalized Euler--Lagrange equation \eqref{gel} is sufficiently
regular (in order to make $\partial_3 L$ and $J^{1-\alpha}_{a+} [\partial_4 L]$ absolutely continuous),
then \eqref{gel} taken in $u$ can be written as the following fractional Euler--Lagrange equation:
\begin{equation}
\dfrac{d}{dt} \big( \partial_3 L \big) - D^{\alpha}_{a+} [\partial_4 L]
= \partial_1 L-J^{1-\alpha}_{a+} [\partial_2 L].
\end{equation}


\section{Some improvements for Section~\ref{section1}}
\label{section4}

In this section, we assume more regularity of the Lagrangian $L$ and of the operator $K$.
It allows to weaken the convexity assumption in Theorem~\ref{thmtonelli}
and/or the assumptions of Propositions~\ref{prop1} and \ref{prop1b}.


\subsection{A first weaker convexity assumption}

Let us assume that $L$ satisfies the following condition:
\begin{equation}
\label{eqequicontcont1}
\Big( L(\cdot,x_2,x_3,x_4,t) \Big)_{(x_2,x_3,x_4,t) \in (\R^d)^{3} \times [a,b]}
\; \text{is uniformly equicontinuous on $\R^d$}.
\end{equation}
This condition has to be understood as:
\begin{multline}
\label{eqequicont1}
\forall \varepsilon > 0, \; \exists \delta > 0,
\; \forall (y,z) \in (\R^d)^2, \; \Vert y-z \Vert
\leq \delta \Longrightarrow \forall (x_2,x_3,x_4,t) \in (\R^d)^{3} \times [a,b], \\
\vert L(y,x_2,x_3,x_4,t) - L(z,x_2,x_3,x_4,t) \vert \leq \varepsilon.
\end{multline}
For example, this condition is satisfied for a Lagrangian $L$ with bounded $\partial_1 L$.
In this case, we can prove the following improved version of Theorem~\ref{thmtonelli}:

\begin{theorem}
\label{thmtonelliimproved1}
Let us assume that:
\begin{itemize}
\item $L$ satisfies the condition given in \eqref{eqequicontcont1};
\item $L$ is regular;
\item $\LL$ is coercive on $\E$;
\item $L(x_1,\cdot,t)$ is convex on $(\R^d)^3$ for any $x_1 \in \R^d$ and for any $t \in [a,b]$.
\end{itemize}
Then, there exists a minimizer for $\LL$.
\end{theorem}

\begin{proof}
Indeed, with the same proof of Theorem~\ref{thmtonelli},
we can construct a weakly convergent sequence $(u_n)_{n \in \N} \subset \E$ satisfying:
\begin{equation}
u_n \xrightharpoonup[]{\W^{1,p}} \bar{u} \in \E \quad \text{and}
\quad \LL (u_n) \longrightarrow \inf\limits_{u \in \E} \LL (u) < +\infty.
\end{equation}
Since the compact embedding $\W^{1,p} \hooktwoheadrightarrow \CC$ holds,
we have $u_n \xrightarrow[]{\CC} \bar{u}$. Let $\varepsilon > 0$
and let us consider $\delta >0$ given in equation \eqref{eqequicont1}.
There exists $N \in \N$ such that for any $n \geq N$,
$\Vert u_n - \bar{u} \Vert_\infty \leq \delta$.
So, for any $n \geq N$ and for almost all $t \in (a,b)$:
\begin{equation}
\vert L(u_n,K[u_n],\dot{u}_n,K[\dot{u}_n],t)
- L(\bar{u},K[u_n],\dot{u}_n,K[\dot{u}_n],t)
\vert \leq \varepsilon.
\end{equation}
Consequently, for any $n \geq N$, we have:
\begin{equation}
\LL (u_n) \geq \di \int_a^b L(\bar{u},K[u_n],\dot{u}_n,K[\dot{u}_n],t) \; dt - (b-a) \varepsilon.
\end{equation}
From the convexity hypothesis and using the same strategy
as in the proof of Theorem~\ref{thmtonelli}, we have by passing to the limit on $n$:
\begin{equation}
\inf\limits_{u \in \E} \LL (u) \geq \LL (\bar{u}) - (b-a) \varepsilon.
\end{equation}
The proof is complete since the previous inequality is true for any $\varepsilon > 0$.
\end{proof}

Such an improvement allows to give examples of a Lagrangian $L$ without convexity
on its first variable. Taking inspiration from Example~\ref{ex1},
we can provide the following example:

\begin{example}
Let us consider $p=2$, $q \geq 2$ and $\E = \W^{1,2}_a$. Let us consider:
\begin{equation}
L(x_1,x_2,x_3,x_4,t) = f(x_1,t)+\dfrac{1}{2} \di \sum_{i=2}^4 \Vert x_i \Vert^2,
\end{equation}
for any function $\fonctionsansdef{f}{\R^d \times [a,b]}{\R}$ of class $\CC^1$
with $\partial_1 f$ bounded (like sine or cosine function). In this case,
$L$ satisfies the hypothesis of Theorem~\ref{thmtonelliimproved1} and we can conclude
with the existence of a minimizer of $\LL$ defined on $\E$.
\end{example}


\subsection{A second weaker convexity assumption}

In this section, we assume that $K$ is moreover a linear bounded operator from $\CC$ to $\CC$.
For example, this condition is satisfied by fractional integrals given
in Sections~\ref{section32} and \ref{section33} (see \cite{kilb,samk} for detailed proofs).
We also assume that $L$ satisfies the following condition:
\begin{equation}
\label{eqequicontcont2}
\Big( L(\cdot,\cdot,x_3,x_4,t) \Big)_{(x_3,x_4,t) \in (\R^d)^{2} \times [a,b]}
\; \text{is uniformly equicontinuous on $(\R^d)^2$}.
\end{equation}
This condition has to be understood as:
\begin{multline}
\label{eqequicont2}
\forall \varepsilon > 0, \; \exists \delta > 0, \; \forall (y,z) \in (\R^d)^2, \;
\forall (y_0,z_0) \in (\R^d)^2, \; \Vert y-z \Vert \leq \delta, \; \Vert y_0-z_0 \Vert
\leq \delta \\ \Longrightarrow \forall (x_3,x_4,t) \in (\R^d)^{2} \times [a,b],
\; \vert L(y,y_0,x_3,x_4,t) - L(z,z_0,x_3,x_4,t) \vert \leq \varepsilon.
\end{multline}
For example, this condition is satisfied for a Lagrangian $L$ with bounded $\partial_1 L$
and bounded $\partial_2 L$. In this case, we can prove
the following improved version of Theorem~\ref{thmtonelli}:

\begin{theorem}
\label{thmtonelliimproved2}
Let us assume that:
\begin{itemize}
\item $L$ satisfies the condition given in \eqref{eqequicontcont2};
\item $L$ is regular;
\item $\LL$ is coercive on $\E$;
\item $L(x_1,x_2,\cdot,t)$ is convex on $(\R^d)^2$
for any $(x_1,x_2) \in (\R^d)^2$ and for any $t \in [a,b]$.
\end{itemize}
Then, there exists a minimizer for $\LL$.
\end{theorem}

\begin{proof}
Indeed, with the same proof of Theorem~\ref{thmtonelli}, we can construct
a weakly convergent sequence $(u_n)_{n \in \N} \subset \E$ satisfying:
\begin{equation}
u_n \xrightharpoonup[]{\W^{1,p}} \bar{u} \in \E \quad \text{and}
\quad \LL (u_n) \longrightarrow \inf\limits_{u \in \E} \LL (u) < +\infty.
\end{equation}
Since the compact embedding $\W^{1,p} \hooktwoheadrightarrow \CC$ holds,
we have $u_n \xrightarrow[]{\CC} \bar{u}$ and since $K$ is continuous from
$\CC$ to $\CC$, we have $K[u_n] \xrightarrow[]{\CC} K[\bar{u}]$.
Let $\varepsilon > 0$ and let us consider $\delta >0$ given
in equation \eqref{eqequicont2}. There exists $N \in \N$ such that for any $n \geq N$,
$\Vert u_n - \bar{u} \Vert_\infty \leq \delta$ and $\Vert K[u_n] - K[\bar{u}] \Vert_\infty
\leq \delta$. So, for any $n \geq N$ and for almost all $t \in (a,b)$:
\begin{equation}
\vert L(u_n,K[u_n],\dot{u}_n,K[\dot{u}_n],t)
- L(\bar{u},K[\bar{u}],\dot{u}_n,K[\dot{u}_n],t)
\vert \leq \varepsilon.
\end{equation}
Consequently, for any $n \geq N$, we have:
\begin{equation}
\LL (u_n) \geq \di \int_a^b L(\bar{u},K[\bar{u}],
\dot{u}_n,K[\dot{u}_n],t) \; dt - (b-a) \varepsilon.
\end{equation}
From the convexity hypothesis and using the same strategy
as in the proof of Theorem~\ref{thmtonelli},
we have by passing to the limit on $n$:
\begin{equation}
\inf\limits_{u \in \E} \LL (u) \geq \LL (\bar{u}) - (b-a) \varepsilon.
\end{equation}
The proof is complete since the previous inequality is true for any $\varepsilon > 0$.
\end{proof}

Such an improvement allows to give examples of a Lagrangian $L$ without convexity
on its two first variables. Taking inspiration from Example~\ref{ex3},
we can provide the following example:

\begin{example}
Let us consider:
\begin{equation}
L(x_1,x_2,x_3,x_4,t) = c(t) \cos (x_1)
\cdot \sin (x_2)+\dfrac{1}{p} \Vert x_3 \Vert^p + f(t) \cdot x_4,
\end{equation}
where $\fonctionsansdef{c}{[a,b]}{\R}$, $\fonctionsansdef{f}{[a,b]}{\R^d}$
are of class $\CC^1$. In this case, one can prove that $L$ satisfies all hypothesis
of Theorem~\ref{thmtonelliimproved2} and then, we can conclude with the existence
of a minimizer of $\LL$ defined on $\W^{1,p}_a$ for any $1 < p < \infty$ and $1 < q < \infty$.
\end{example}


\subsection{First weaker assumptions in Propositions~\ref{prop1} and \ref{prop1b}}

In this section, we assume that $K$ is moreover a linear bounded operator from
$\CC$ to $\CC$. This hypothesis implies that for any $u \in \W^{1,p}$, $K[u] \in \CC$.
Let us remind that such an assumption is satisfied by fractional integrals of Riemann--Liouville. \\

Consequently, let us define the set $\PP^1_M$ of maps
$\fonctionsansdef{P}{(\R^d)^4 \times [a,b]}{\R^+}$ such that
for any $(x_1,x_2,x_3,x_4,t) \in (\R^d)^4 \times [a,b]$:
\begin{equation}
P(x_1,x_2,x_3,x_4,t) = \di \sum_{k=0}^{N} c_k(x_1,x_2,t)
\Vert x_3 \Vert^{d_{3,k}} \Vert x_4 \Vert^{d_{4,k}},
\end{equation}
with $ N \in \N$ and where, for any $k=0,\ldots,N$,
$\fonctionsansdef{c_k}{(\R^d)^2 \times [a,b]}{\R^+}$ is continuous
and $(d_{3,k},d_{4,k}) \in [0,p] \times [0,q]$ satisfies
$(q/p) d_{3,k} + d_{4,k} \leq (q/M)$. \\

From these new sets of maps, one can prove the following
improved versions of Propositions~\ref{prop1} and \ref{prop1b}:
\begin{itemize}
\item Proposition~\ref{prop1} with the weaker assumption $\PP^1_{M}$ instead of $\PP_{M}$;
\item Proposition~\ref{prop1b} with the weaker assumption $(q/p)d_{3,k}+d_{4,k} \leq q$
instead of $d_{2,k}+(q/p)d_{3,k}+d_{4,k} \leq q$.
\end{itemize}


\subsection{Second weaker assumptions in Propositions~\ref{prop1} and \ref{prop1b}}

In this section, we assume that $K$ is a linear bounded operator from $\L^p$ to $\CC$.
Let us remind that such an assumption is satisfied by fractional integrals
of Riemann--Liouville in the case $\alpha > (1/p)$, see detailed proof in \cite{bour6}.
This hypothesis implies that for any $u \in \W^{1,p}$, $K[u] \in \CC$ and $K[\dot{u}] \in \CC$. \\

Consequently, let us define the set $\PP^2_M$ of maps $\fonctionsansdef{P}{(\R^d)^4
\times [a,b]}{\R^+}$ such that for any $(x_1,x_2,x_3,x_4,t) \in (\R^d)^4 \times [a,b]$:
\begin{equation}
P(x_1,x_2,x_3,x_4,t) = \di \sum_{k=0}^{N} c_k(x_1,x_2,x_4,t) \Vert x_3 \Vert^{d_{3,k}},
\end{equation}
with $ N \in \N$ and where, for any $k=0,\ldots,N$,
$\fonctionsansdef{c_k}{(\R^d)^3 \times [a,b]}{\R^+}$
is continuous and $0 \leq d_{3,k} \leq p$.\\

From these new sets of maps, one can prove the following
improved versions of Propositions~\ref{prop1} and \ref{prop1b}:
\begin{itemize}
\item Proposition~\ref{prop1} with the weaker assumption $\PP^2_{M}$ instead of $\PP_{M}$;
\item Proposition~\ref{prop1b} with the weaker assumption $d_{3,k} \leq p$
instead of $d_{2,k}+(q/p)d_{3,k}+d_{4,k} \leq q$.
\end{itemize}


\section{Conclusion and perspectives}
\label{section5}

In this paper, the operator $K$ is devoted to be a kernel operator.
Nevertheless, one can use the results of Sections~\ref{section1}
and \ref{section2} for any linear operator bounded from $\L^p$ to $\L^q$.


\subsection{Example of a general operator $K$ which is not a kernel}

For instance, one can consider the following substitution operator:
\begin{equation}
\forall f \in \L^p, \; K[f] = f \circ \varphi,
\end{equation}
where $\varphi$ is a $\CC^1$-diffeomorphism on the interval $[a,b]$
satisfying $\varphi (a) = a$ and $\varphi (b) = b$. In this case, one can easily
prove that $K$ is linear and bounded from $\L^p$
to $\L^p$ and its adjoint operator is given by:
\begin{equation}
\forall f \in \L^{p'}, \; K^*[f] = (f \circ \varphi^{-1})/ \dot{\varphi}.
\end{equation}


\subsection{Extension of the method used in this paper}

In this work, we have extended the results of \cite{bour6,bour9}
from fractional Lagrangian functionals to generalized ones. \\

In the same way, although we have generalized our existence result,
it can not cover all the possible Lagrangians and all the possible operators $K$.
Nevertheless, in most of cases, the method can be applied in its whole picture.
For proving the existence of a minimizer for a particular variational problem,
one has just to improve this method with respect to the particular case in question. \\

We end this paper with the following remark.
This method can be applied in many other variational problems:
\begin{itemize}
\item with higher order derivatives;
\item with different operators $K_1$, $K_2$, \ldots;
\item in the multidimensional case;
\item on time scales (and in particular in the discrete-time case).
\end{itemize}
Of course, this is a non exhaustive list of generalizing perspectives
where the method used in this paper can be developed.


\section*{Acknowledgements}

This work is part of the first author's Ph.D. project, carried out at the Universit\'{e}
de Pau et des Pays de l'Adour, under the scientific supervision of J. Cresson and I. Greff.
The second and third authors are grateful to the support of the Center for Research
and Development in Mathematics and Applications (CIDMA), University of Aveiro, Portugal,
and The Portuguese Foundation for Science and Technology (FCT),
within project PEst-C/MAT/UI4106/2011 with COMPETE number FCOMP-01-0124-FEDER-022690.



\end{document}